\documentclass[12pt,a4paper,english]{smfart}
\usepackage{smfenum}
\usepackage{amsmath,amssymb,amsthm,latexsym,MnSymbol}
\usepackage[utf8]{inputenc}
\usepackage{url}
\usepackage[pdftex]{graphicx}
\usepackage{tikz,tikz-cd}
\usepackage[pdftex]{hyperref}
\usepackage[all]{xy}
\usepackage[margin=2cm]{geometry}

\newcommand{\Z}{\mathbf{Z}}
\newcommand{\Q}{\mathbf{Q}}

\newcommand{\R}{\mathbf{R}}
\newcommand{\C}{\mathbf{C}}
\newcommand{\h}{\mathcal{H}}
\newcommand{\dd}{\mathrm{d}}

\DeclareMathOperator{\Cl}{Cl}

\DeclareMathOperator{\darg}{darg}
\DeclareMathOperator{\dv}{div}

\DeclareMathOperator{\GL}{GL}

\DeclareMathOperator{\ord}{ord}

\DeclareMathOperator{\SL}{SL}

\newtheorem{thm}{Theorem}
\newtheorem{lem}[thm]{Lemma}
\newtheorem{pro}[thm]{Proposition}

\theoremstyle{definition}
\newtheorem{definition}[thm]{Definition}

\theoremstyle{remark}
\newtheorem{remark}[thm]{Remark}

\author[F. Brunault]{François Brunault}
\address{ÉNS Lyon, UMPA\\ 46 allée d'Italie\\ 69007 Lyon, France}
\email{francois.brunault@ens-lyon.fr}
\urladdr{http://perso.ens-lyon.fr/francois.brunault}

\title[Regulators of Siegel units]{Regulators of Siegel units and applications}

\begin{document}
\renewcommand{\refname}{References}
\frontmatter

\begin{abstract}
We compute the regulator of two Siegel units and give applications to Boyd's conjectures and Zagier's conjectures.
\end{abstract}

\subjclass{11F03, 11F11, 11F67, 11G16, 19F27}
\keywords{}

\maketitle

\mainmatter

In a recent work \cite{zudilin}, W. Zudilin proved a formula for the regulator of two modular units. The aim of this article is to generalize this result to arbitrary Siegel units and give applications to elliptic curves.

For any two holomorphic functions $f$ and $g$ on a Riemann surface, define the real 1-form
\begin{equation*}
\eta(f,g) = \log |f| \, \dd \arg g - \log |g| \, \dd \arg f.
\end{equation*}
Note that if $f$ and $g$ are two modular units, then $\eta(f,g)$ is a well-defined 1-form on the upper half-plane $\h$. We prove the following theorem.

\begin{thm}\label{main thm}
Let $N \geq 1$ be an integer. Let $u=(a,b)$ and $v=(c,d)$ be two nonzero vectors in $(\Z/N\Z)^2$, and let $g_u$ and $g_v$ be the Siegel units associated to $u$ and $v$ (see \S\ref{sec siegel} for the definition). We have
\begin{equation}\label{main formula}
\int_0^{i\infty} \eta(g_u,g_v)  =\pi \Lambda^*(e_{a,d} e_{b,-c}+e_{a,-d} e_{b,c},0)
\end{equation}
where $e_{a,b}$ is the Eisenstein series of weight $1$ and level $N^2$ defined by
\begin{equation}
e_{a,b}(\tau) = \alpha_0(a,b) + \sum_{ \substack{m,n \geq 1\\ m \equiv a, \; n \equiv b (N)}} q^{mn} - \sum_{ \substack{m,n \geq 1\\ m \equiv -a, \; n \equiv -b (N)}} q^{mn} \qquad (q=e^{2\pi i \tau})
\end{equation}
with
\begin{equation*}
\alpha_0(a,b) = \begin{cases} 0 & \textrm{if } a=b=0\\
\frac12 - \{\frac{b}{N}\} & \textrm{if } a=0 \textrm{ and } b \neq 0\\
\frac12 - \{\frac{a}{N}\} & \textrm{if } a \neq 0 \textrm{ and } b = 0\\
0 & \textrm{if } a \neq 0 \textrm{ and } b \neq 0.
\end{cases}
\end{equation*}
Here $\Lambda^*(f,0)$ denotes the regularized value of the completed $L$-function $\Lambda(f,s)$ at $s=0$ (see \S\ref{sec Lstar} for the definition).
\end{thm}

\begin{remark}
Let $Y(N)$ be the affine modular curve of level $N$ over $\Q$, and let $X(N)$ be the smooth compactification of $Y(N)$. The homology group $H=H_1(X(N)(\C),\{\mathrm{cusps}\},\Z)$ is generated by the modular symbols $\xi(\gamma) = \{\gamma 0, \gamma \infty\}$ with $\gamma \in \GL_2(\Z/N\Z)$. If $c$ is an element of $H$, we may write $c = \sum_i \lambda_i \xi(\gamma_i)$, and it follows that
\begin{equation*}
\int_c \eta(g_u,g_v) = \sum_i \lambda_i \int_0^\infty \eta(g_{u \gamma_i},g_{v \gamma_i}).
\end{equation*}
Moreover, the Siegel units generate (up to constants) the group $\mathcal{O}(Y(N))^\times \otimes \Q$. Therefore Theorem \ref{main thm} (together with Lemma \ref{lem int darg}) gives a formula for all possible regulator integrals $\int_c \eta(u,v)$ with $c \in H$ and $u,v \in \mathcal{O}(Y(N))^\times$.
\end{remark}

\begin{remark}
Theorem \ref{main thm} is a generalization of \cite[Thm 1]{zudilin}. More precisely, let $\tilde{g}_a$, $a \in \Z/N\Z$ denote the modular units arising in \cite{zudilin}. Then for every $c \in \Z/N\Z$, we have
\begin{equation*}
\tilde{g}_a(c/N+it) = g_{a,ac}(iNt) \qquad (t>0).
\end{equation*}
We recover \cite[Thm 1]{zudilin} by taking $u=(a,ac)$ and $v=(b,bc)$ in Theorem \ref{main thm}. Note that in this case $f_{a,b,c}=e_{a,bc}e_{ac,-b}+e_{a,-bc}e_{ac,b}$ belongs to $\Q[[q^N]]$, and $f_{a,b,c}(\tau/N)$ is a modular form on $\Gamma_1(N)$. More generally, if $M = \begin{pmatrix} a & b \\ c & d \end{pmatrix} \in M_2(\Z/N\Z)$ is any matrix such that $\det (M)=0$, then $F_M(\tau/N) = (e_{a,d} e_{b,-c} + e_{a,-d} e_{b,c})(\tau/N)$ is a modular form of weight $2$ on $\Gamma_1(N)$. It would be interesting to study further the properties of these modular forms and to understand their possible relations with the toric modular forms introduced by L. Borisov and P. Gunnells \cite{borisov-gunnells1}.
\end{remark}

The proof of Theorem \ref{main thm} follows the strategy of \cite{zudilin}. We express the logarithms of Siegel units as a double infinite sum (Lemma \ref{lem gu}) and deduce an expression for the regulator as a quadruple sum. We then perform the same analytical change of variables from \cite{rogers-zudilin}, leading to the Mellin transform of a product of Eisenstein series. The key lemma to do this (Lemma \ref{lem fgh}) suggests that similar results should hold in higher weight.

We point out that the simple shape of the Eisenstein series $e_{a,b}$ makes Theorem \ref{main thm} particularly amenable to explicit computations. We give some applications of Theorem \ref{main thm} in Section \ref{sec applications}, for elliptic curves which are parametrized by modular units \cite{brunault:mod_units}.

I would like to thank Wadim Zudilin for helpful exchanges related to this work.

\section{Siegel units}\label{sec siegel}

We recall some basic definitions and results about Siegel units, for which we refer the reader to \cite[\S 1]{kato} and \cite{kubert-lang}.

Let $B_2 = X^2-X+\frac16$ be the second Bernoulli polynomial. For $x \in \R$, we define $B(x) = B_2(\{x\}) = \{x\}^2-\{x\}+\frac16$, where $\{x\} = x- \lfloor x \rfloor$ denotes the fractional part of $x$.

Let $\h$ be the upper half-plane. Let $N \geq 1$ be an integer and $\zeta_N=e^{2\pi i/N}$. For any $(a,b) \in (\Z/N\Z)^2$, $(a,b) \neq (0,0)$, the Siegel unit $g_{a,b}$ on $\h$ is defined by
\begin{equation}\label{def gab}
g_{a,b}(\tau) = q^{B(a/N)/2} \prod_{n \geq 0} (1-q^n q^{\tilde{a}/N} \zeta_N^b) \prod_{n \geq 1} (1-q^n q^{-\tilde{a}/N} \zeta_N^{-b}) \qquad (q=e^{2\pi i \tau})
\end{equation}
where $\tilde{a}$ is the representative of $a$ satisfying $0 \leq \tilde{a} <N$. Here $q^\alpha = e^{2\pi i \alpha \tau}$ for $\alpha \in \Q$. It is known that the function $g_{a,b}^{12N}$ is modular for the group
\begin{equation*}
\Gamma(N) = \{ \gamma \in \SL_2(\Z) : \gamma \equiv I_2 \pmod{N} \}.
\end{equation*}
In fact $g_{a,b}$ defines an element of $\mathcal{O}(Y(N))^\times \otimes \Q$, where $Y(N)$ denotes the affine modular curve of level $N$ over $\Q$. Recall that the group $\GL_2(\Z/N\Z)$ acts on $Y(N)$ by $\Q$-automorphisms. For any $\gamma \in \GL_2(\Z/N\Z)$, we have the identity in $\mathcal{O}(Y(N))^\times \otimes \Q$
\begin{equation}\label{gab gamma}
g_{a,b} | \gamma = g_{(a,b)\gamma}.
\end{equation}

\begin{lem}\label{lem gab sigma}
Let $(a,b) \in (\Z/N\Z)^2$, $(a,b) \neq (0,0)$. We have
\begin{equation}
g_{a,b}(-1/\tau) = e^{-2\pi i (\{\frac{a}{N}\}-\frac12)(\{\frac{b}{N}\}-\frac12)} g_{b,-a}(\tau) \qquad (\tau \in \h).
\end{equation}
\end{lem}

\begin{proof}
By taking the matrix $\gamma = \begin{pmatrix} 0 & -1 \\ 1 & 0 \end{pmatrix}$ in (\ref{gab gamma}), we see that $g_{a,b}(-1/\tau)=w_{a,b}g_{b,-a}(\tau)$ for some root of unity $w_{a,b}$. The formula for $w_{a,b}$ follows from \cite[Chap. 2, \S 1, K1, K4]{kubert-lang}.
\end{proof}

\begin{lem}\label{lem int darg}
For any $a,b \in \Z/N\Z$, we have
\begin{equation}\label{int darg eq}
\int_0^\infty \darg g_{a,b} = \begin{cases} 0 & \textrm{if } a=0 \textrm{ or } b=0\\
2\pi (\{\frac{a}{N}\}-\frac12)(\{\frac{b}{N}\}-\frac12) & \textrm{if } a \neq 0 \textrm{ and } b \neq 0.
\end{cases}
\end{equation}
\end{lem}

\begin{proof}
If $a=0$ or $b=0$ then $g_{a,b}$ has constant argument on the imaginary axis $\tau = it$, $t >0$, hence $\int_0^\infty \darg g_{a,b}=0$.

If $a \neq 0$ and $b \neq 0$, it is easily seen that $\arg g_{a,b}(it) \xrightarrow{t \to \infty} 0$. Moreover, by Lemma \ref{lem gab sigma}, we have $\arg g_{a,b}(it) \xrightarrow{t \to 0} -2\pi (\{\frac{a}{N}\}-\frac12)(\{\frac{b}{N}\}-\frac12) \pmod{2\pi}$. This proves (\ref{int darg eq}) up to a multiple of $2\pi$. In order to establish the exact equality, let us introduce the Klein forms \cite[Chap. 2, \S 1, p. 27]{kubert-lang}:
\begin{equation*}
\mathfrak{k}_{\alpha,\beta}(\tau) = e^{-\frac12 \eta(\alpha \tau+\beta,\tau) (\alpha \tau+\beta)} \sigma(\alpha \tau+\beta, \tau) \qquad (\alpha,\beta \in \R; \tau \in \h)
\end{equation*}
where $\eta$ and $\sigma$ denote the Weierstrass functions. The link with Siegel units is given by
\begin{equation*}
g_{a,b}(\tau) = w \mathfrak{k}_{a/N,b/N}(\tau) \Delta(\tau)^{1/12} \qquad (1 \leq a,b \leq N-1)
\end{equation*}
where $w$ is a root of unity \cite[p. 29]{kubert-lang}. Since $\Delta$ is positive on the imaginary axis, it follows that
\begin{equation*}
\int_0^\infty \darg g_{a,b} = \int_0^\infty \darg \mathfrak{k}_{a/N,b/N}.
\end{equation*}
Using the $q$-product formula for the $\sigma$ function \cite[Chap. 18, \S 2]{lang:elliptic} and the Legendre relation $\eta_2 \omega_1 - \eta_1 \omega_2 = 2\pi i$, we find
\begin{equation}\label{eq kab}
\mathfrak{k}_{\alpha,\beta}(it)=\frac{1}{2\pi i} e^{-\pi \alpha^2 t} e^{\pi i\alpha \beta}(e^{\pi i \beta}e^{-\pi \alpha t}-e^{-\pi i \beta} e^{\pi \alpha t}) \prod_{n \geq 1} \frac{(1-e^{-2\pi (n+\alpha)t}e^{2\pi i \beta})(1-e^{-2\pi (n-\alpha)t}e^{-2\pi i \beta})}{(1-e^{-2\pi nt})^2}.
\end{equation}
Assume $0 < \alpha,\beta < 1$. Then by (\ref{eq kab}), we have $\arg \mathfrak{k}_{\alpha,\beta}(it) \xrightarrow{t \to \infty} \pi (\alpha \beta-\beta+\frac12)$. Moreover, the Klein forms are homogeneous of weight -1 \cite[p. 27, K1]{kubert-lang}, which implies
\begin{equation*}
\mathfrak{k}_{\alpha,\beta}(-1/\tau) = \frac{1}{\tau} \mathfrak{k}_{\beta,-\alpha}(\tau).
\end{equation*}
From this we get $\arg \mathfrak{k}_{\alpha,\beta}(it) \xrightarrow{t \to 0} \pi (-\alpha \beta+\alpha) \pmod{2\pi}$ and
\begin{equation*}
\int_0^\infty \darg \mathfrak{k}_{\alpha,\beta} \equiv 2\pi (\alpha-\frac12)(\beta-\frac12) \pmod{2\pi}.
\end{equation*}
Moreover, using the fact that $\int_0^\infty \darg \mathfrak{k}_{\alpha,\beta} = \int_i^{\infty} \darg \mathfrak{k}_{\alpha,\beta} - \int_i^\infty \darg \mathfrak{k}_{\beta,-\alpha}$ and taking the imaginary part of the logarithm of (\ref{eq kab}), we may express $\int_0^\infty \darg \mathfrak{k}_{\alpha,\beta}$ as an infinite sum, which shows that it is a continuous function of $(\alpha,\beta) \in (0,1)^2$. But for $\beta=\frac12$, the Klein form $\mathfrak{k}_{\alpha,\frac12}(it)$ has constant argument. This implies that $\int_0^\infty \darg \mathfrak{k}_{\alpha,\beta} = 2\pi (\alpha-\frac12)(\beta-\frac12)$ for any $0 < \alpha,\beta < 1$.
\end{proof}

\section{$L$-functions of modular forms}\label{sec Lstar}

In this section we recall basic results on the functional equation satisfied by $L$-functions of modular forms.

Let $f(\tau) = \sum_{n=0}^\infty a_n q^n$ be a modular form of weight $k \geq 1$ on the group $\Gamma_1(N)$. The $L$-function of $f$ is defined by $L(f,s)=\sum_{n=1}^\infty a_n n^{-s}$, $\Re(s)>k$. Define the completed $L$-function
\begin{equation*}
\Lambda(f,s):=N^{s/2} (2\pi)^{-s} \Gamma(s) L(f,s) = N^{s/2} \int_0^\infty (f(iy)-a_0) y^s \frac{\dd y}{y}.
\end{equation*}

Recall that the Atkin-Lehner involution $W_N$ on $M_k(\Gamma_1(N))$ is defined by $(W_N f)(\tau)=i^k N^{-k/2} \tau^{-k} f(-1/(N\tau))$ (note that in the case $k=2$ this $W_N$ is the opposite of the usual involution acting on differential $1$-forms). The following theorem is classical (see \cite[Thm 4.3.5]{miyake}).

\begin{thm}\label{thm hecke}
Let $f=\sum_{n=0}^\infty a_n q^n \in M_k(\Gamma_1(N))$. The function $\Lambda(f,s)$ can be analytically continued to the whole $s$-plane, and satisfies the functional equation $\Lambda(f,s) = \Lambda(W_N f,k-s)$. Moreover, write $W_N f = \sum_{n=0}^\infty b_n q^n$. Then the function
\begin{equation*}
\Lambda(f,s)+\frac{a_0}{s}+\frac{b_0}{k-s}
\end{equation*}
is holomorphic on the whole $s$-plane.
\end{thm}

\begin{definition}
The notations being as in Theorem \ref{thm hecke}, we define the regularized values of $\Lambda(f,s)$ at $s=0$ and $s=k$ by
\begin{align}
\Lambda^*(f,0) & := \lim_{s \to 0} \Lambda(f,s)+\frac{a_0}{s}\\
\Lambda^*(f,k) & := \lim_{s \to k} \Lambda(f,s)+\frac{b_0}{k-s}.
\end{align}
\end{definition}

Note that the functional equation translates into the equalities of regularized values
\begin{equation}\label{eq lambda star}
\Lambda^*(f,0) = \Lambda^*(W_N f,k) \qquad \Lambda^*(f,k) = \Lambda^*(W_N f,0).
\end{equation}

We will need the following lemma.

\begin{lem}\label{lem fgh}
Let $f = \sum_{n=0}^\infty a_n q^n \in M_k(\Gamma_1(N))$ and $g = \sum_{n=0}^\infty b_n q^n \in M_{\ell}(\Gamma_1(N))$ with $k,\ell \geq 1$. Let $h=W_N(g)$. Write $f^* = f-a_0$ and $g^* = g-b_0$. Then for any $s \in \C$, we have
\begin{equation}\label{eq fgh}
N^{s/2} \int_0^\infty f^*(iy) g^*\bigl(\frac{i}{Ny}\bigr) y^s \frac{\dd y}{y} = \Lambda(fh,s+\ell) - a_0 \Lambda(h,s+\ell) - b_0 \Lambda(f,s).
\end{equation}
\end{lem}

\begin{proof}
Note that the integral in (\ref{eq fgh}) is absolutely convergent because $f^*(\tau)$ and $g^*(\tau)$ have exponential decay when $\Im(\tau)$ tends to $+\infty$. Moreover, it is easy to check, using Theorem \ref{thm hecke}, that the right hand side of (\ref{eq fgh}) is holomorphic on the whole $s$-plane. Therefore it suffices to establish (\ref{eq fgh}) when $\Re(s)>k$. Since $W_N g=h$, we have
\begin{align*}
N^{s/2} \int_0^\infty f^*(iy) g^*\bigl(\frac{i}{Ny}\bigr) y^s \frac{\dd y}{y} & = N^{s/2} \int_0^\infty f^*(iy) \bigl(g\bigl(\frac{i}{Ny}\bigr)-b_0\bigr) y^s \frac{\dd y}{y}\\
& = N^{s/2} \int_0^\infty f^*(iy) (N^{\ell/2} t^\ell h(iy)-b_0) y^s \frac{\dd y}{y}.
\end{align*}
Now, we remark that $f^* h = fh-a_0 h = (fh)^*-a_0 h^*$. Thus
\begin{align*}
N^{s/2} \int_0^\infty f^*(iy) g^*\bigl(\frac{i}{Ny}\bigr) y^s \frac{\dd y}{y} & = N^{s/2} \int_0^\infty \bigl(N^{\ell/2} y^\ell ((fh)^*(iy) -a_0 h^*(iy)) - b_0 f^*(iy) \bigr) y^s \frac{\dd y}{y}\\
& = \Lambda(fh,s+\ell) - a_0 \Lambda(h,s+\ell) - b_0 \Lambda(f,s).
\end{align*}
\end{proof}

Specializing Lemma \ref{lem fgh} to the (regularized) value at $s=k$, we get the following formula.

\begin{lem}\label{lem fgh 2}
Let $f = \sum_{n=0}^\infty a_n q^n \in M_k(\Gamma_1(N))$ and $g = \sum_{n=0}^\infty b_n q^n \in M_{\ell}(\Gamma_1(N))$ with $k,\ell \geq 1$. Let $h=W_N(g)$. Write $f^* = f-a_0$ and $g^* = g-b_0$. Then we have
\begin{equation}\label{eq fgh 2}
N^{k/2} \int_0^\infty f^*(iy) g^*\bigl(\frac{i}{Ny}\bigr) y^k \frac{\dd y}{y} = \Lambda^*(fh,k+\ell) - a_0 \Lambda(h,k+\ell) - b_0 \Lambda^*(f,k).
\end{equation}
\end{lem}

\section{Eisenstein series of weight 1}

In this section we define some Eisenstein series of weight 1. These are the same as those arising in \cite{zudilin}.

\begin{definition}
For any $a,b \in \Z/N\Z$, we let
\begin{equation}
e_{a,b} = \alpha_0(a,b) + \sum_{ \substack{m,n \geq 1\\ m \equiv a, \; n \equiv b (N)}} q^{mn} - \sum_{ \substack{m,n \geq 1\\ m \equiv -a, \; n \equiv -b (N)}} q^{mn}
\end{equation}
where
\begin{equation*}
\alpha_0(a,b) = \begin{cases} 0 & \textrm{if } a=b=0\\
\frac12 - \{\frac{b}{N}\} & \textrm{if } a=0 \textrm{ and } b \neq 0\\
\frac12 - \{\frac{a}{N}\} & \textrm{if } a \neq 0 \textrm{ and } b = 0\\
0 & \textrm{if } a \neq 0 \textrm{ and } b \neq 0.
\end{cases}
\end{equation*}
\end{definition}

\begin{lem}
The function $e_{a,b}(\tau/N)$ is an Eisenstein series of weight $1$ on the group $\Gamma(N)$, and the function $e_{a,b}$ is an Eisenstein series of weight $1$ on $\Gamma_1(N^2)$.
\end{lem}

\begin{proof}
In \cite[Chap. VII, \S 2.3]{schoeneberg}, for any $(a,b) \in (\Z/N\Z)^2$ the following Eisenstein series are introduced
\begin{equation*}
G_{1,(a,b)}(\tau) = -\frac{2\pi i}{N} \Bigl(\gamma_0(a,b)+\sum_{\substack{m,n \geq 1\\ n \equiv a (N)}} \zeta_N^{bm} q^{mn/N} - \sum_{\substack{m,n \geq 1\\ n \equiv -a (N)}} \zeta_N^{-bm} q^{mn/N}\Bigr)
\end{equation*}
where
\begin{equation*}
\gamma_0(a,b) = \begin{cases} 0 & \textrm{if } a=b=0\\
\frac12 \frac{1+\zeta_N^b}{1-\zeta_N^b} & \textrm{if } a=0 \textrm{ and } b \neq 0\\
\frac12 - \{\frac{a}{N}\} & \textrm{if } a \neq 0.
\end{cases}
\end{equation*}
The function $G_{1,(a,b)}$ is an Eisenstein series of weight $1$ on the group $\Gamma(N)$. We have
\begin{align*}
e_{a,b}\left(\frac{\tau}{N}\right) & = \alpha_0(a,b) + \sum_{\substack{m,n \geq 1 \\ m \equiv a, \; n \equiv b (N)}} q^{mn/N} - \sum_{\substack{m,n \geq 1 \\ m \equiv -a, \; n \equiv -b (N)}} q^{mn/N}\\
& = \alpha_0(a,b)+ \frac{1}{N} \sum_{c=0}^{N-1} \zeta_N^{ca} \Biggl( \sum_{\substack{m,n \geq 1 \\ n \equiv b (N)}} \zeta_N^{-cm} q^{mn/N} - \sum_{\substack{m,n \geq 1 \\ n \equiv -b (N)}} \zeta_N^{cm} q^{mn/N} \Biggr)\\
& = \alpha_0(a,b)-\frac{1}{N} \sum_{c=0}^{N-1} \zeta_N^{ca} \gamma_0(b,-c) -\frac{1}{2\pi i} \sum_{c=0}^{N-1} \zeta_N^{ca} G_{1,(b,-c)}.
\end{align*}
If $b \neq 0$ then
\begin{equation*}
\frac{1}{N} \sum_{c=0}^{N-1} \zeta_N^{ca} \gamma_0(b,-c) = \frac{1}{N} \sum_{c=0}^{N-1} \zeta_N^{ca} \bigl(\frac12 - \{\frac{b}{N}\}\bigr) = \alpha_0(a,b),
\end{equation*}
hence $e_{a,b}(\tau/N)$ is an Eisenstein series of weight $1$ on $\Gamma(N)$. If $a \neq 0$ then the same is true because $e_{a,b}=e_{b,a}$. Finally if $a=b=0$ then
\begin{equation*}
\alpha_0(a,b)-\frac{1}{N} \sum_{c=0}^{N-1} \zeta_N^{ca} \gamma_0(b,-c) =-\frac{1}{N} \sum_{c=0}^{N-1} \gamma_0(0,c) = 0
\end{equation*}
because $\gamma_0(0,-c)=-\gamma_0(0,c)$.

The second assertion follows from the fact that $\begin{pmatrix} N & 0 \\ 0 & 1 \end{pmatrix} \Gamma_1(N^2) \begin{pmatrix} N & 0 \\ 0 & 1 \end{pmatrix}^{-1} \subset \Gamma(N)$.
\end{proof}

\begin{definition}
For any $a,b \in \Z/N\Z$, we let
\begin{equation}
f_{a,b} = \beta_0(a,b) + \sum_{m,n \geq 1} (\zeta_N^{am+bn}-\zeta_N^{-am-bn}) q^{mn}
\end{equation}
where
\begin{equation*}
\beta_0(a,b) = \begin{cases} 0 & \textrm{if } a=b=0\\
\frac12 \frac{1+\zeta_N^b}{1-\zeta_N^b} & \textrm{if } a=0 \textrm{ and } b \neq 0\\
\frac12 \frac{1+\zeta_N^a}{1-\zeta_N^a} & \textrm{if } a \neq 0 \textrm{ and } b = 0\\
\frac12 \Bigl( \frac{1+\zeta_N^a}{1-\zeta_N^a} +  \frac{1+\zeta_N^b}{1-\zeta_N^b}\Bigr) & \textrm{if } a \neq 0 \textrm{ and } b \neq 0.
\end{cases}
\end{equation*}
\end{definition}

As the next lemma shows, the functions $f_{a,b}$ are also Eisenstein series; they relate to $e_{a,b}$ by the Atkin-Lehner involution of level $N^2$.

\begin{lem}\label{eab sigma}
We have the relation
\begin{equation}\label{eab fab}
e_{a,b}\left(-\frac{1}{N\tau}\right) = -\frac{\tau}{N} f_{a,b}\left(\frac{\tau}{N}\right) \qquad (\tau \in \h).
\end{equation}
The function $f_{a,b}(\tau/N)$ is an Eisenstein series of weight $1$ on $\Gamma(N)$, and the function $f_{a,b}$ is an Eisenstein series of weight $1$ on $\Gamma_1(N^2)$. Moreover, we have $W_{N^2} (e_{a,b}) = -\frac{i}{N} f_{a,b}$.
\end{lem}

\begin{proof}
The relation (\ref{eab fab}) follows from \cite[Lemma 2]{zudilin} (the proof there works for arbitrary $a,b \in \Z/N\Z$). We deduce that $f_{a,b}(\tau/N)$ is a multiple of the function obtained from $e_{a,b}(\tau/N)$ by applying the slash operator $| \begin{pmatrix} 0 & -1 \\ 1 & 0 \end{pmatrix}$ in weight 1. Hence $f_{a,b}(\tau/N)$ is an Eisenstein series of weight $1$ on $\Gamma(N)$. The last assertion follows from replacing $\tau$ by $N\tau$ in (\ref{eab fab}).
\end{proof}

We will need the following formula for the completed $L$-function of $f_{a,b}$.

\begin{lem}\label{lem Lfab}
For any $a,b \in \Z/N\Z$, we have
\begin{equation}\label{eq Lfab}
\Lambda(f_{a,b}+f_{-a,b},s) = N^s \Gamma(s) (2\pi)^{-s} \Bigl(\sum_{m \geq 1} \frac{\zeta_N^{am}+\zeta_N^{-am}}{m^s}\Bigr) \Bigl(\sum_{n \geq 1} \frac{\zeta_N^{bn}-\zeta_N^{-bn}}{n^s}\Bigr).
\end{equation}
\end{lem}

\begin{proof}
See the proof of \cite[Lemma 3]{zudilin}.
\end{proof}

In the special cases $s=1$ and $s=2$, this gives the following formulas. Note that formula (\ref{eq Lfab2}) is none other than \cite[Lemma 3]{zudilin}.

\begin{lem} \label{lem Lfab12}
We have
\begin{align}
\label{eq Lfab1} \Lambda^*(f_{a,b}+f_{-a,b},1) & = \begin{cases} 0 & \textrm{if } b=0\\
2iN \gamma \cdot (\frac12-\{\frac{b}{N}\}) & \textrm{if } a=0 \textrm{ and } b \neq 0\\
-2iN \log |1-\zeta_N^a| \cdot (\frac12-\{\frac{b}{N}\}) & \textrm{if } a \neq 0 \textrm{ and } b \neq 0
\end{cases}\\
\label{eq Lfab2} \Lambda(f_{a,b}+f_{-a,b},2) & = iN^2 B\bigl(\frac{a}{N}\bigr) \Cl_2\bigl(\frac{2\pi b}{N}\bigr)
\end{align}
where $\gamma$ is Euler's constant and
\begin{equation*}
\Cl_2(x) = \sum_{m=1}^\infty \frac{\sin(mx)}{m^2} \qquad (x \in \R)
\end{equation*}
denotes the Clausen dilogarithmic function.
\end{lem}

\begin{proof}
If $a=0$ then $\sum_{n=1}^\infty \zeta_N^{an} n^{-s}=\zeta(s)=\frac{1}{s-1}+\gamma+O_{s \to 1}(s-1)$. If $a \neq 0$ then $\sum_{n=1}^\infty \zeta_N^{an}/n = -\log(1-\zeta_N^a)$ where we use the principal value of the logarithm. Formula (\ref{eq Lfab1}) follows, noting that $-\log \frac{1-\zeta_N^b}{1-\zeta_N^{-b}} = 2\pi i (\frac12-\{\frac{b}{N}\})$. Formula (\ref{eq Lfab2}) is \cite[Lemma 3]{zudilin}.
\end{proof}

\section{The computation}

\begin{lem}\label{lem gu}
For any $(a,b) \in (\Z/N\Z)^2$, $(a,b) \neq (0,0)$, we have
\begin{align}
\label{gu1} \log g_{a,b}(it) & = -\pi B(a/N) t + C_{a,b} - \sum_{m \geq 1} \sum_{\substack{n \geq 1 \\ n \equiv a (N)}} \frac{\zeta_N^{bm}}{m} e^{-\frac{2\pi mnt}{N}} - \sum_{m \geq 1} \sum_{\substack{n \geq 1 \\ n \equiv -a (N)}} \frac{\zeta_N^{-bm}}{m} e^{-\frac{2\pi mnt}{N}}\\
\label{gu2} & = -\frac{\pi B(b/N)}{t} + C_{b,-a} + i \theta_{a,b} - \sum_{m \geq 1} \sum_{\substack{n \geq 1 \\ n \equiv b (N)}} \frac{\zeta_N^{-am}}{m} e^{-\frac{2\pi mn}{Nt}} - \sum_{m \geq 1} \sum_{\substack{n \geq 1 \\ n \equiv -b (N)}} \frac{\zeta_N^{am}}{m} e^{-\frac{2\pi mn}{Nt}}
\end{align}
where $\theta_{a,b}=2\pi (\{\frac{a}{N}\}-\frac12)(\{\frac{b}{N}\}-\frac12)$ and
\begin{equation}
C_{a,b} = \begin{cases} \log(1-\zeta_N^b) & \textrm{if } a=0,\\
 0 & \textrm{if } a \neq 0.
\end{cases}
\end{equation}
\end{lem}

\begin{proof}
By the definition of Siegel units, we have
\begin{equation*}
\log g_{a,b} = \pi i B(a/N) \tau+ \sum_{n \geq 0} \log(1-q^n q^{\tilde{a}/N} \zeta_N^b) + \sum_{n \geq 1} \log (1-q^n q^{-\tilde{a}/N} \zeta_N^{-b})
\end{equation*}
Using the identity $\log(1-x)=-\sum_{m=1}^{\infty} \frac{x^m}{m}$ and substituting $\tau=it$, we get (\ref{gu1}). Applying Lemma \ref{lem gab sigma} with $\tau=i/t$, we have $g_{a,b}(it)=e^{i\theta_{a,b}} g_{b,-a}(i/t)$, whence (\ref{gu2}).
\end{proof}

We will need the following lemma from \cite{zudilin}.

\begin{lem}\cite[Lemma 4]{zudilin}\label{lem int dlog}
For any $a,b \in \Z/N\Z$, we have
\begin{equation}
\begin{split}
I(a,b) := & \int_0^\infty \frac{1}{it} \dd \sum_{m=1}^\infty \frac{\zeta_N^{am}-\zeta_N^{-am}}{m} \Biggl( \sum_{\substack{n \geq 1 \\ n \equiv b (N)}} - \sum_{\substack{n \geq 1 \\ n \equiv -b (N)}} \Biggr) \exp\bigl(-\frac{2\pi mn}{Nt}\bigr)\\
& \qquad = \begin{cases} 0 & \textrm{if } a=0 \textrm{ or } b=0\\
-i \Cl_2(\frac{2\pi a}{N}) \frac{1+\zeta_N^b}{1-\zeta_N^b} & \textrm{if } a \neq 0 \textrm{ and } b \neq 0.
\end{cases}
\end{split}
\end{equation}
\end{lem}

\begin{proof}[Proof of Theorem \ref{main thm}]
By Lemma \ref{lem gu}, we get
\begin{equation}\label{proof eq 1}
\log |g_u(it)| = -\frac{\pi B(b/N)}{t} + \Re(C_{b,-a}) - \frac12 \sum_{m \geq 1} \frac{\zeta_N^{am}+\zeta_N^{-am}}{m} \Biggl( \sum_{\substack{n \geq 1 \\ n \equiv b (N)}} + \sum_{\substack{n \geq 1 \\ n \equiv -b (N)}} \Biggr) e^{-\frac{2\pi mn}{Nt}}
\end{equation}
and
\begin{align}
\label{proof eq 2} \darg g_u(it) & = -\frac{1}{2i} \dd \sum_{m \geq 1} \frac{\zeta_N^{bm}-\zeta_N^{-bm}}{m} \Biggl( \sum_{\substack{n \geq 1 \\ n \equiv a (N)}} - \sum_{\substack{n \geq 1 \\ n \equiv -a (N)}} \Biggr) e^{-\frac{2\pi mnt}{N}}\\
\label{proof eq 3} & = \frac{1}{2i} \dd \sum_{m \geq 1}  \frac{\zeta_N^{am}-\zeta_N^{-am}}{m} \Biggl( \sum_{\substack{n \geq 1 \\ n \equiv b (N)}} -  \sum_{\substack{n \geq 1 \\ n \equiv -b (N)}} \Biggr) e^{-\frac{2\pi mn}{Nt}}.
\end{align}
Let $u=(a,b)$, $v=(c,d) \in (\Z/N\Z)^2$, $u,v \neq (0,0)$. We have
\begin{equation}\label{proof eq 4}
\begin{split}
\eta(g_u,g_v) & = \Bigl(-\frac{\pi B(b/N)}{t}+\Re(C_{b,-a}) \Bigr) \cdot \frac{1}{2i} \dd \sum_{m \geq 1}  \frac{\zeta_N^{cm}-\zeta_N^{-cm}}{m} \Biggl( \sum_{\substack{n \geq 1 \\ n \equiv d (N)}} -  \sum_{\substack{n \geq 1 \\ n \equiv -d (N)}} \Biggr) e^{-\frac{2\pi mn}{Nt}}\\
& \quad - \frac12 \sum_{m_1 \geq 1} \frac{\zeta_N^{am_1}+\zeta_N^{-am_1}}{m_1} \Biggl( \sum_{\substack{n_1 \geq 1 \\ n_1 \equiv b (N)}} + \sum_{\substack{n_1 \geq 1 \\ n_1 \equiv -b (N)}} \Biggr) e^{-\frac{2\pi m_1 n_1}{Nt}}\\
& \quad \quad \times -\frac{1}{2i} \dd \sum_{m_2 \geq 1} \frac{\zeta_N^{dm_2}-\zeta_N^{-dm_2}}{m_2} \Biggl( \sum_{\substack{n_2 \geq 1 \\ n_2 \equiv c (N)}} - \sum_{\substack{n_2 \geq 1 \\ n_2 \equiv -c (N)}} \Biggr) e^{-\frac{2\pi m_2 n_2 t}{N}}\\
& \quad - \Bigl(-\frac{\pi B(d/N)}{t}+\Re(C_{d,-c}) \Bigr) \cdot \frac{1}{2i} \dd \sum_{m \geq 1}  \frac{\zeta_N^{am}-\zeta_N^{-am}}{m} \Biggl( \sum_{\substack{n \geq 1 \\ n \equiv b (N)}} -  \sum_{\substack{n \geq 1 \\ n \equiv -b (N)}} \Biggr) e^{-\frac{2\pi mn}{Nt}}\\
& \quad + \frac12 \sum_{m_1 \geq 1} \frac{\zeta_N^{cm_1}+\zeta_N^{-cm_1}}{m_1} \Biggl( \sum_{\substack{n_1 \geq 1 \\ n_1 \equiv d (N)}} + \sum_{\substack{n_1 \geq 1 \\ n_1 \equiv -d (N)}} \Biggr) e^{-\frac{2\pi m_1 n_1}{Nt}}\\
& \quad \quad \times -\frac{1}{2i} \dd \sum_{m_2 \geq 1} \frac{\zeta_N^{bm_2}-\zeta_N^{-bm_2}}{m_2} \Biggl( \sum_{\substack{n_2 \geq 1 \\ n_2 \equiv a (N)}} - \sum_{\substack{n_2 \geq 1 \\ n_2 \equiv -a (N)}} \Biggr) e^{-\frac{2\pi m_2 n_2 t}{N}}.
\end{split}
\end{equation}
The terms involving double sums can be integrated using Lemmas \ref{lem int darg} and \ref{lem int dlog}. This gives
\begin{equation}\label{proof eq 4bis}
\begin{split}
\int_0^\infty \eta(g_u,g_v) & = -\frac{\pi}{2} B\bigl(\frac{b}{N}\bigr) I(c,d) + \frac{\pi}{2} B\bigl(\frac{d}{N}\bigr) I(a,b)\\
& \quad +\Re(C_{b,-a}) \int_0^\infty \darg g_v - \Re(C_{d,-c}) \int_0^\infty \darg g_u + I
\end{split}
\end{equation}
with
\begin{equation}\label{proof eq 5}
\begin{split}
I & = \frac{\pi i}{2N} \sum_{m_1,m_2 \geq 1} \Biggl( (\zeta_N^{am_1}+\zeta_N^{-am_1}) (\zeta_N^{dm_2}-\zeta_N^{-dm_2}) \Biggl( \sum_{\substack{n_1 \geq 1 \\ n_1 \equiv b (N)}} + \sum_{\substack{n_1 \geq 1 \\ n_1 \equiv -b (N)}} \Biggr) \Biggl( \sum_{\substack{n_2 \geq 1 \\ n_2 \equiv c (N)}} - \sum_{\substack{n_2 \geq 1 \\ n_2 \equiv -c (N)}} \Biggr) \\
& \qquad - (\zeta_N^{cm_1}+\zeta_N^{-cm_1}) (\zeta_N^{bm_2}-\zeta_N^{-bm_2}) \Biggl( \sum_{\substack{n_1 \geq 1 \\ n_1 \equiv d (N)}} + \sum_{\substack{n_1 \geq 1 \\ n_1 \equiv -d (N)}} \Biggr) \Biggl( \sum_{\substack{n_2 \geq 1 \\ n_2 \equiv a (N)}} - \sum_{\substack{n_2 \geq 1 \\ n_2 \equiv -a (N)}} \Biggr)\Biggr) \cdot\\
& \qquad \cdot \frac{n_2}{m_1} \int_0^\infty \exp \left(-2\pi \left(\frac{m_1 n_1}{Nt}+\frac{m_2 n_2 t}{N}\right)\right) \dd t.
\end{split}
\end{equation}
Making the change of variables $t'=\frac{n_2}{m_1}t$, we have
\begin{equation}\label{proof eq 6}
\frac{n_2}{m_1} \int_0^\infty \exp \left(-2\pi \left(\frac{m_1 n_1}{Nt}+\frac{m_2 n_2 t}{N}\right)\right) \dd t = \int_0^\infty \exp \left(-2\pi \left(\frac{n_1 n_2}{Nt'}+\frac{m_1 m_2 t'}{N}\right)\right) \dd t'.
\end{equation}
Replacing in (\ref{proof eq 5}) and interchanging integral and summation, we get
\begin{equation}\label{proof eq 7}
\begin{split}
I = \frac{\pi i}{2N} \int_0^\infty & \sum_{m_1,m_2 \geq 1} (\zeta_N^{am_1}+\zeta_N^{-am_1}) (\zeta_N^{dm_2}-\zeta_N^{-dm_2}) e^{-\frac{2\pi m_1 m_2 t'}{N}} \cdot \\
& \qquad \cdot \Biggl( \sum_{\substack{n_1 \geq 1 \\ n_1 \equiv b (N)}} + \sum_{\substack{n_1 \geq 1 \\ n_1 \equiv -b (N)}} \Biggr) \Biggl( \sum_{\substack{n_2 \geq 1 \\ n_2 \equiv c (N)}} - \sum_{\substack{n_2 \geq 1 \\ n_2 \equiv -c (N)}} \Biggr) e^{-\frac{2\pi n_1 n_2}{Nt'}}\\
& - \sum_{m_1, m_2 \geq 1} (\zeta_N^{cm_1}+\zeta_N^{-cm_1}) (\zeta_N^{bm_2}-\zeta_N^{-bm_2}) e^{-\frac{2\pi m_1 m_2 t'}{N}} \cdot \\
& \qquad \cdot \Biggl( \sum_{\substack{n_1 \geq 1 \\ n_1 \equiv d (N)}} + \sum_{\substack{n_1 \geq 1 \\ n_1 \equiv -d (N)}} \Biggr) \Biggl( \sum_{\substack{n_2 \geq 1 \\ n_2 \equiv a (N)}} - \sum_{\substack{n_2 \geq 1 \\ n_2 \equiv -a (N)}} \Biggr)\Biggr) e^{-\frac{2\pi n_1 n_2}{Nt'}} \dd t'.
\end{split}
\end{equation}
Making the change of variables $y=t'/N$, we obtain
\begin{equation}\label{proof eq 8}
\begin{split}
I = \frac{\pi i}{2} \int_0^\infty & (f^*_{a,d}+f^*_{-a,d})(iy) \cdot (e^*_{b,c}+e^*_{-b,c}) \bigl(\frac{i}{N^2 y}\bigr)\\
& - (f^*_{c,b}+f^*_{-c,b})(iy) \cdot (e^*_{d,a}+e^*_{-d,a}) \bigl(\frac{i}{N^2 y}\bigr) \dd y.
\end{split}
\end{equation}
We compute this integral using Lemma \ref{lem fgh 2} with $k=\ell=1$, taking into account Lemma \ref{eab sigma}: for any $a,b,c,d \in \Z/N\Z$, we have
\begin{equation}\label{proof eq 9}
\int_0^\infty f^*_{a,b}(iy) e^*_{c,d}\bigl(\frac{i}{N^2 y}\bigr) \dd y = -\frac{i}{N^2} \bigl(\Lambda^*(f_{a,b} f_{c,d},2) - \beta_0(a,b) \Lambda(f_{c,d},2) \bigr) -\frac{\alpha_0(c,d)}{N} \Lambda^*(f_{a,b},1).
\end{equation}
Replacing in (\ref{proof eq 8}), we get $I=I_1+I_2+I_3$ with
\begin{align}
\label{eq I1} I_1 & = \frac{\pi}{2N^2} \Lambda^*((f_{a,d}+f_{-a,d})(f_{b,c}+f_{-b,c})-(f_{c,b}+f_{-c,b})(f_{d,a}+f_{-d,a}),2)\\
\label{eq I2} I_2 & = - \frac{\pi}{2N^2} \Bigl((\beta_0(a,d)+\beta_0(-a,d)) \Lambda(f_{b,c}+f_{-b,c},2) - (\beta_0(c,b)+\beta_0(-c,b)) \Lambda(f_{d,a}+f_{-d,a},2)\Bigr)\\
\label{eq I3} I_3 & = -\frac{\pi i}{2N} \Bigl((\alpha_0(b,c)+\alpha_0(-b,c)) \Lambda^*(f_{a,d}+f_{-a,d},1) - (\alpha_0(d,a)+\alpha_0(-d,a)) \Lambda^*(f_{c,b}+f_{-c,b},1)\Bigr)
\end{align}
Using the fact that $f_{a,b}=f_{b,a}=-f_{-a,-b}$, $I_1$ simplifies to
\begin{equation}\label{eq I1 2}
I_1 = \frac{\pi}{N^2} \Lambda^*(f_{a,d} f_{-b,c} - f_{a,-d} f_{b,c},2).
\end{equation}
The terms involving $\Lambda(f_{a,b},2)$ can be evaluated with (\ref{eq Lfab2}); they simplify with the terms involving $I(a,b)$ in (\ref{proof eq 4bis}):
\begin{equation}\label{eq I2 2}
I_2 = \frac{\pi}{2} B\bigl(\frac{b}{N}\bigr) I(c,d) - \frac{\pi}{2} B\bigl(\frac{d}{N}\bigr) I(a,b).
\end{equation}
The terms involving $\Lambda^*(f_{a,b},1)$ can be evaluated with (\ref{eq Lfab1}). Note that $\alpha_0(b,c)+\alpha_0(-b,c)$ is nonzero only in the case $b=0$ and $c \neq 0$. Since we assumed $u \neq 0$, this implies $a \neq 0$ and the case of Lemma \ref{lem Lfab12} involving Euler's constant does not happen. Anyway $I_3$ simplifies with the terms involving $\int_0^\infty \darg g_u$ in (\ref{proof eq 4bis}):
\begin{equation}\label{eq I3 2}
I_3 = -\Re(C_{b,-a}) \int_0^\infty \darg g_v + \Re(C_{d,-c}) \int_0^\infty \darg g_u.
\end{equation}
Putting everything together, we get
\begin{equation}
\int_0^\infty \eta(g_u,g_v) = I_1 = \frac{\pi}{N^2} \Lambda^* (f_{a,d} f_{-b,c}-f_{a,-d} f_{b,c},2).
\end{equation}
Theorem \ref{main thm} now follows from (\ref{eq lambda star}), taking into account the fact that $W_{N^2}(f_{a,b}f_{c,d}) = W_{N^2}(f_{a,b}) W_{N^2}(f_{c,d}) = -N^2 e_{a,b} e_{c,d}$.
\end{proof}

\section{Applications}\label{sec applications}

In this section we investigate the applications of Theorem \ref{main thm} to elliptic curves. Our strategy can be explained as follows. In \cite{brunault:mod_units}, we determined a list of elliptic curves defined over $\Q$ which can be parametrized by modular units. Let $E$ be such an elliptic curve, with modular parametrization $\varphi : X_1(N) \to E$. Let $x,y$ be functions on $E$ such that $u:=\varphi^*(x)$ and $v:=\varphi^*(y)$ are modular units. Assume that $\{x,y\} \in K_2(E) \otimes \Q$. Then the minimal polynomial $P$ of $(x,y)$ is tempered and in favorable cases, the Mahler measure of $P$ can be expressed in terms of a regulator integral $\int_\gamma \eta(x,y)$ where $\gamma$ is a (non necessarily closed) path on $E$. Using the techniques of \cite{brunault:mod_units}, we compute the images of the various cusps under $\varphi$ and deduce the divisors of $u$ and $v$. Since the divisors of Siegel units are easily computed using (\ref{def gab}) and (\ref{gab gamma}), we get an expression of $u$ and $v$ in terms of Siegel units, and may apply Theorem \ref{main thm}.

We will need the following expression for the regulator integral in terms of Bloch's elliptic dilogarithm. Let $E/\Q$ be an elliptic curve, and let $D_E : E(\C) \to \R$ be the elliptic dilogarithm associated to a chosen orientation of $E(\R)$. Extend $D_E$ by linearity to a function $\Z[E(\C)] \to \R$. Let $\gamma_E^+$ be the generator of $H_1(E(\C),\Z)^+$ corresponding to the chosen orientation.

\begin{pro}\label{pro int eta DE}
Let $x \in K_2(E) \otimes \Q$. Choose rational functions $f_i,g_i$ on $E$ such that $x = \sum_i \{f_i,g_i\}$, and define $\eta(x) = \sum_i \eta(f_i,g_i)$. Then for every $\gamma \in H_1(E(\C),\Z)$, we have
\begin{equation*}
\int_\gamma \eta(x) = -(\gamma_E^+ \bullet \gamma) D_E(\beta)
\end{equation*}
where $\bullet$ denotes the intersection product on $H_1(E(\C),\Z)$, and $\beta$ is the divisor given by
\begin{equation*}
\beta = \sum_i \sum_{p,q \in E(\C)} \ord_p(f_i) \ord_q(g_i) (p-q).
\end{equation*}
\end{pro}

\begin{proof}
Since $x \in K_2(E) \otimes \Q$, the integral of $\eta(x)$ over a closed path $\gamma$ avoiding the zeros and poles of $f_i,g_i$ depends only on the class of $\gamma$ in $H_1(E(\C),\Z)$. Let $\delta$ be an element of $H_1(E(\C),\Z)$ such that $\gamma_E^+ \bullet \delta = 1$. Let $c$ denote the complex conjugation on $E(\C)$. Since $c^* \eta(x) = -\eta(x)$, we have $\int_{\gamma_E^+} \eta(x)=0$ and it suffices to prove the formula for $\gamma=\delta$. Choose an isomorphism $E(\C) \cong \C/(\Z+\tau\Z)$ which is compatible with complex conjugation. We have
\begin{equation*}
\overline{\int_{E(\C)} \eta(x) \wedge \dd z} = \int_{E(\C)} \eta(x) \wedge \dd \overline{z} = \int_{E(\C)} c^* ( -\eta(x) \wedge \dd z) = \int_{E(\C)} \eta(x) \wedge \dd z
\end{equation*}
so that $\int_{E(\C)} \eta(x) \wedge \dd z \in \R$. By \cite[Prop. 6]{brunault:LEF}, we get
\begin{equation*}
\int_{E(\C)} \eta(x) \wedge \dd z = D_E(\beta).
\end{equation*}
Since $(\gamma_E^+,\delta)$ is a symplectic basis of $H_1(E(\C),\Z)$, we have \cite[A.2.5]{bost}
\begin{equation*}
\int_{E(\C)} \eta(x) \wedge \dd z = \int_{\gamma_E^+} \eta(x) \cdot \int_\delta \dd z - \int_{\gamma_E^+} \dd z \cdot \int_\delta \eta(x) = -\int_\delta \eta(x).
\end{equation*}
\end{proof}

The following proposition is a slight generalization of a technique introduced by A. Mellit \cite{mellit} to prove identities involving elliptic dilogarithms. Let $E/\Q$ be an elliptic curve, which we view as a smooth cubic in $\mathbf{P}^2$.

\begin{definition}
For any lines $\ell$ and $m$ in $\mathbf{P}^2$, let $\beta_E(\ell,m)$ be the divisor of degree $9$ on $E(\C)$ defined by $\beta_E(\ell,m) = \sum_{x \in \ell \cap E} \sum_{y \in m \cap E} (x-y)$.
\end{definition}

\begin{pro}\label{pro incident}
Let $\ell_1,\ell_2,\ell_3$ be three incident lines in $\mathbf{P}^2$. Then
\begin{equation}\label{eq incident}
D_{E}(\beta_E(\ell_1,\ell_2))+D_{E}(\beta_E(\ell_2,\ell_3))+D_{E}(\beta_E(\ell_3,\ell_1))=0.
\end{equation}
\end{pro}

\begin{proof}
Let $f_1,f_2,f_3$ be equations of $\ell_1,\ell_2,\ell_3$ such that $f_1+f_2=f_3$. Using the Steinberg relation $\{\frac{f_1}{f_3},\frac{f_2}{f_3}\}=0$, we deduce $\{f_1,f_2\}+\{f_2,f_3\}+\{f_3,f_1\}=0$ in $K_2(\C(E)) \otimes \Q$. Applying the regulator map and taking the real part \cite[Prop. 6]{brunault:LEF}, we deduce
\begin{equation*}
D_E(\beta(f_1,f_2))+D_E(\beta(f_2,f_3))+D_E(\beta(f_3,f_1))=0
\end{equation*}
where $\beta(f_i,f_{i+1})$ is defined as in Proposition \ref{pro int eta DE}. We have $\dv(f_i)=(\ell_i \cap E) - 3(0)$ so that
\begin{equation*}
\beta(f_i,f_{i+1})=\beta_E(\ell_i,\ell_{i+1})-3(\ell_i \cap E) - 3 \iota^* (\ell_{i+1} \cap E) + 9(0)
\end{equation*}
where $\iota$ denotes the map $p \mapsto -p$ on $E(\C)$. Since $D_E$ is odd, the proposition follows.
\end{proof}

\begin{remark}
If the incidence point of $\ell_1,\ell_2,\ell_3$ lies on $E$, then the relation (\ref{eq incident}) is trivial in the sense that it is a consequence of the fact that $D_E$ is odd.
\end{remark}

We will also need the following lemma to relate elliptic dilogarithms on isogenous curves.

\begin{lem}\label{lem DE DE'}
Let $\varphi : E \to E'$ be an isogeny between elliptic curves defined over $\Q$. Choose orientations of $E(\R)$ and $E'(\R)$ which are compatible under $\varphi$, and let $d_\varphi$ be the topological degree of the map $E(\R)^0 \to E'(\R)^0$, where $(\cdot)^0$ denotes the connected component of the origin. Then for any point $P' \in E'(\C)$, we have
\begin{equation}\label{eq DE DE'}
D_{E'}(P') = d_\varphi \cdot \sum_{\varphi(P)=P'} D_E(P).
\end{equation}
\end{lem}

\begin{proof}
Choose isomorphisms $E(\C) \cong \C/(\Z+\tau\Z)$ and $E'(\C) \cong \C/(\Z+\tau'\Z)$ which are compatible with complex conjugation. Then $E(\R)=\R/\Z$ and $E'(\R) = \R/\Z$ so that $\varphi$ is given by $[z] \mapsto [d_\varphi z]$. We have isomorphisms $E(\C) \cong \C^\times /q^\Z$ and $E'(\C) \cong \C^\times / (q')^\Z$ with $q=e^{2\pi i \tau}$ and $q'=e^{2\pi i \tau'}$. Let $\pi : \C^\times \to E(\C)$ and $\pi' : \C^\times \to E'(\C)$ be the canonical maps. Let $P'$ be a point of $E'(\C)$. By definition $D_{E'}(P') = \sum_{\pi'(x')=P'} D(x')$ where $D$ is the Bloch-Wigner function, and similarly $D_E(P) = \sum_{\pi(x)=P} D(x)$. Now $\varphi$ is induced by the map $x \mapsto x^{d_\varphi}$, so that (\ref{eq DE DE'}) follows from the usual functional equation $D(x^r) = r \sum_{u^r=1} D(ux)$ for any $r \geq 1$ \cite[(21)]{oesterle}.
\end{proof}

Note that in the particular case $\varphi$ is the multiplication-by-$n$ map on $E$, Lemma \ref{lem DE DE'} gives the usual functional equation
\begin{equation*}
D_E(nP) = n \sum_{Q \in E[n]} D_E(P+Q).
\end{equation*}

\subsection{Conductors 14, 35 and 54}

We prove the following cases of Boyd's conjectures \cite[Table 5, $k=-1,-2,-3$]{boyd:expmath}. Note that the case of conductor 14 was proved by A. Mellit \cite{mellit}.

\begin{thm}
Let $P_k$ be the polynomial $P_k(x,y)=y^2+kxy+y-x^3$, and let $E_k$ be the elliptic curve defined by the equation $P_k(x,y)=0$. We have the identities
\begin{align}
\label{mP-1} m(P_{-1}) & = 2 L'(E_{-1},0)\\
\label{mP-2} m(P_{-2}) & = L'(E_{-2},0)\\
\label{mP-3} m(P_{-3}) & =L'(E_{-3},0).
\end{align}
\end{thm}

By the discussion in \cite[p. 62]{boyd:expmath}, the polynomial $P_k$ does not vanish on the torus for $k \in \R$, $k <-1$. For these values of $k$ we thus have
\begin{equation*}
m(P_k) = \frac{1}{2\pi} \int_{\gamma_k} \eta(x,y)
\end{equation*}
where $\gamma_k$ is the closed path on $E_k(\C)$ defined by
\begin{equation*}
\gamma_k = \{(x,y) \in E_k(\C) : |x|=1, |y| \leq 1\}.
\end{equation*}
The point $A=(0,0)$ on $E_k$ has order $3$ and the divisors of $x$ and $y$ are given by
\begin{equation*}
\dv(x) = (A)+(-A)-2(0) \qquad \dv(y) = 3(A)-3(0).
\end{equation*}
The tame symbols of $\{x,y\}$ at $0$, $A$, $-A$ are respectively equal to $1,-1,-1$, so that $\{x,y\}$ defines an element of $K_2(E_k) \otimes \Q$. Moreover $\gamma_k$ is a generator of $H_1(E_k(\C),\Z)^-$ which satisfies $\gamma_{E_k}^+ \bullet \gamma_k = -2$, so that Proposition \ref{pro int eta DE} gives
\begin{equation}\label{mPk DEkA}
m(P_k) = \frac{1}{\pi} D_{E_k}(\beta(x,y)) = \frac{9}{\pi} D_{E_k}(A) \qquad (k<-1).
\end{equation}
Note that by continuity (\ref{mPk DEkA}) also holds for $k=-1$.

Now assume $k \in \{-1,-2,-3\}$. The elliptic curves $E_{-1}$, $E_{-2}$, $E_{-3}$ are respectively isomorphic to $14a4$, $35a3$ and $54a3$. By \cite{brunault:mod_units}, these curves are parametrized by modular units. Since the functions $x$ and $y$ are supported in the rational torsion subgroup, their pull-back $u=\varphi^* x$ and $v=\varphi^* y$ are modular units, and we may express them in terms of Siegel units. For brevity, we put $g_b = g_{0,b}$ in what follows. We also let $f_{-k}$ be the newform associated to $E_{-k}$.

In the case $k=-1$, $N=14$, we find explicitly
\begin{equation*}
u  = \frac{g_{5} g_{6}}{g_{1} g_{2}} \qquad v = -\frac{g_{3} g_{5} g_{6}^2}{g_{1}^2 g_{2} g_{4}}.
\end{equation*}
Moreover the Deninger path is the following sum of modular symbols
\begin{equation*}
\gamma_{-1} = \varphi_* \left\{\frac27,-\frac27\right\} = \varphi_* \left( -\xi \begin{pmatrix} 2 & 1 \\ 7 & 4 \end{pmatrix}-\xi \begin{pmatrix} 1 & 0 \\ 4 & 1 \end{pmatrix}+\xi \begin{pmatrix} 1 & 0 \\ -4 & 1 \end{pmatrix} + \xi \begin{pmatrix} -2 & 1 \\ 7 & -4 \end{pmatrix} \right).
\end{equation*}
Using Theorem \ref{main thm}, we obtain
\begin{equation*}
\int_{\gamma_{-1}} \eta(x,y) = \int_{2/7}^{-2/7} \eta(u,v) = \pi L'( 4f_{-1},0).
\end{equation*}
This proves (\ref{mP-1}).

In the case $k=-2$, $N=35$, we find explicitly
\begin{equation*}
u = \frac{g_{2} g_{9} g_{12} g_{15} g_{16}}{g_{3} g_{4} g_{10} g_{11} g_{17}} \qquad v = - \frac{g_{2}^2 g_{5} g_{9}^2 g_{12}^2 g_{15} g_{16}^2}{g_{1} g_{3} g_{4} g_{6} g_{8} g_{10}^2 g_{11} g_{13} g_{17}}.
\end{equation*}
Moreover the Deninger path is the following sum of modular symbols
\begin{equation*}
\gamma_{-2} = \varphi_* \left\{\frac15,-\frac15\right\} = \varphi_* \left( \xi \begin{pmatrix} 1 & 0 \\ -5 & 1 \end{pmatrix}-\xi \begin{pmatrix} 1 & 0 \\ 5 & 1 \end{pmatrix} \right).
\end{equation*}
Using Theorem \ref{main thm}, we obtain
\begin{equation*}
\int_{\gamma_{-2}} \eta(x,y) = \int_{1/5}^{-1/5} \eta(u,v) = \pi L'( 2f_{-2},0).
\end{equation*}
This proves (\ref{mP-2}).

In the case $k=-3$, $N=54$, we find explicitly
\begin{equation*}
x = \frac{g_2 g_4 g_5^2 g_{13}^2 g_{14} g_{16} g_{20} g_{21} g_{22} g_{23}^2 g_{24}}{g_1 g_7 g_8^2 g_{10}^2 g_{11} g_{12} g_{15} g_{17} g_{19} g_{25} g_{26}^2} \qquad y = - \frac{g_2^3 g_3 g_5^3 g_{13}^3 g_{16}^3 g_{20}^3 g_{21} g_{23}^3 g_{24}^2}{g_1^3 g_6 g_8^3 g_{10}^3 g_{12} g_{15}^2 g_{17}^3 g_{19}^3 g_{26}^3}.
\end{equation*}
Moreover the Deninger path is the following sum of modular symbols
\begin{equation*}
\gamma_{-3} = \varphi_* \left\{-\frac18,\frac18\right\} = \varphi_* \left( \xi \begin{pmatrix} 1 & 0 \\ 8 & 1 \end{pmatrix}-\xi \begin{pmatrix} 1 & 0 \\ -8 & 1 \end{pmatrix} \right).
\end{equation*}
Using Theorem \ref{main thm}, we obtain
\begin{equation*}
\int_{\gamma_{-3}} \eta(x,y) = \int_{-1/8}^{1/8} \eta(u,v) = \pi L'( 2f_{-3},0).
\end{equation*}
This proves (\ref{mP-2}).


Using (\ref{mPk DEkA}), we also deduce Zagier's conjectures for these elliptic curves.

\begin{thm}
We have the identities
\begin{equation}
L(E_{-1},2) = \frac{9\pi}{7} D_{E_{-1}}(A) \qquad L(E_{-2},2) = \frac{36\pi}{35} D_{E_{-2}}(A) \qquad L(E_{-3},2) = \frac{2\pi}{3} D_{E_{-3}}(A).
\end{equation}
\end{thm}

\subsection{Conductor 21}

The modular curve $X_0(21)$ has genus $1$ and is isomorphic to the elliptic curve $E_0 = 21a1$ with minimal equation $y^2+xy = x^3-4x-1$. The Mordell-Weil group $E_0(\Q)$ is isomorphic to $\Z/4\Z \times \Z/2\Z$ and is generated by the points $P=(5,8)$ and $Q=(-2,1)$, with respective orders 4 and 2. The modular curve $X_0(21)$ has 4 cusps: $0$, $1/3$, $1/7$, $\infty$ and we may choose the isomorphism $\varphi_0 : X_0(21) \xrightarrow{\cong} E_0$ so that $\varphi_0(0)=0$, $\varphi_0(1/3)=(-1,-1)=P+Q$, $\varphi_0(1/7)=Q$ and $\varphi_0(\infty)=P$. Let $f_P$ and $f_Q$ be functions on $E$ with divisors
\begin{equation*}
(f_P) = 4(P)-4(0) \qquad (f_Q) = 2(Q)-2(0).
\end{equation*}
These modular units can be expressed in terms of the Dedekind $\eta$ function \cite[\S 3.2]{ligozat}:
\begin{equation*}
f_P \sim_{\Q^\times} \frac{\eta(3\tau) \eta(21\tau)^5}{\eta(\tau)^5 \eta(7\tau)} \qquad f_Q \sim_{\Q^\times} \frac{\eta(3\tau) \eta(7\tau)^3}{\eta(\tau)^3 \eta(21\tau)}.
\end{equation*}
They can in turn be expressed in terms of Siegel units using the formula
\begin{equation*}
\frac{\eta(d\tau)}{\eta(\tau)} = C_d \prod_{k=1}^{(d-1)/2} g_{0,kN/d}(\tau) \qquad (C_d \in \C^\times).
\end{equation*}
Thus we can take
\begin{equation*}
f_P = \frac{ g_{0,7} (\prod_{b=1}^{10} g_{0,b})^5}{g_{0,3} g_{0,6} g_{0,9}} \qquad f_Q = \frac{ g_{0,7} (g_{0,3} g_{0,6} g_{0,9})^3}{ \prod_{b=1}^{10} g_{0,b}}.
\end{equation*}
The homology group $H_1(E_0(\C),\Z)^-$ is generated by the modular symbol $\gamma = \{-\frac13,\frac13\} = \xi\begin{pmatrix} 1 & 0 \\ 3 & 1 \end{pmatrix} - \xi\begin{pmatrix} 1 & 0 \\ -3 & 1 \end{pmatrix}$. Using Theorem \ref{main thm} and a computer algebra system, we find
\begin{equation*}
\int_{\gamma} \eta(f_P,f_Q) = \pi \Lambda^*(F,0)
\end{equation*}
where $F$ is the modular form of weight 2 and level 21 given by
\begin{equation*}
F = 68q + 220q^2 + 68q^3 + 508q^4 + 440q^5 + 220q^6 + 508q^7 + 1068q^8 + 68q^9 + \cdots
\end{equation*}
The space $M_2(\Gamma_0(21))$ has dimension 4 and is generated by $f_0$, $E_{2,3}$, $E_{2,7}$ and $E_{2,21}$ where $f_0$ is the newform associated to $E_0$ and $E_{2,d}(\tau) = E_2(\tau)-dE_2(d\tau)$. We find explicitly
\begin{equation*}
F = -4f_0 + 72 E_{2,3} + \frac{72}{7} E_{2,7} - \frac{72}{7} E_{2,21}
\end{equation*}
We have $L(E_{2,d},s) = (1-d^{1-s}) \zeta(s) \zeta(s-1)$ and a little computation gives
\begin{equation*}
L(F,s) = -4L(E_0,s) + \frac{72}{7 \cdot 21^s} (7 \cdot 21^s-21\cdot 7^s - 7 \cdot 3^s + 21) \zeta(s) \zeta(s-1).
\end{equation*}
Thus $L(F,0)=0$ and using $\zeta(0)=-1/2$ and $\zeta(-1)=-1/12$, we find
\begin{equation*}
\Lambda^*(F,0) = \Lambda(F,0) = L'(F,0) = -4L'(E_0,0)-6 \log 7.
\end{equation*}
The extraneous term $6 \log 7$ stems from the fact that the Milnor symbol $\{f_P,f_Q\}$ does not extend to $K_2(E_0) \otimes \Q$. Indeed, the tame symbols are given by
\begin{equation*}
\partial_0 \{f_P,f_Q\} = 1 \qquad \partial_P \{f_P,f_Q\} = f_Q(P)^{-4} = \zeta_7^{-4} 7^{-4} \qquad \partial_Q \{f_P,f_Q\} = \zeta_7^{-4} 7^4.
\end{equation*}
Since $f_P$ and $f_Q$ are supported in torsion points, there is a standard trick (due to Bloch) to alter the symbol $\{f_P,f_Q\}$ to make an element of $K_2(E_0) \otimes \Q$. We will see that the corresponding regulator integral is proportional to $L'(E_0,0)$ alone. We put $x:=\{f_P,f_Q\}+\{7,f_P/f_Q^2\}$, which belongs to $K_2(E_0) \otimes \Q$, and we define
\begin{equation*}
\eta(x):=\eta(f_P,f_Q)+\eta(7,f_P/f_Q^2) = \eta(f_P,f_Q)+\log 7 \cdot \darg(f_P/f_Q^2).
\end{equation*}
We can compute the integral of $\darg(f_P/f_Q^2)$ using Lemma \ref{lem int darg}, which results in
\begin{equation*}
\int_\gamma \eta(x) = -4\pi L'(E_0,0).
\end{equation*}
On the other hand, we have $\int_\gamma \omega_{f_0} \sim 1.91099i$ which shows that $\gamma_{E_0}^+ \bullet \gamma >0$. Since $E_0(\R)$ has two connected components, this implies $\gamma_{E_0}^+ \bullet \gamma=1$ and Proposition \ref{pro int eta DE} gives
\begin{equation*}
\int_\gamma \eta(x) = -D_{E_0}(\beta).
\end{equation*}
We have $\beta = 8(P+Q)-8(P)-8(Q)+8(0)$. Since $D_{E_0}$ is odd, this gives
\begin{equation*}
\int_\gamma \eta(x) = -8 \bigl(D_{E_0}(P+Q)-D_{E_0}(P)\bigr).
\end{equation*}
Taking into account the functional equation $L'(E_0,0) = \frac{21}{4\pi^2} L(E_0,2)$, we have thus shown Zagier's conjecture for $E_0$.

\begin{thm}\label{zagier 21}
We have the identity $L(E_0,2) = \frac{8\pi}{21} \bigl(D_{E_0}(P+Q)-D_{E_0}(P)\bigr)$.
\end{thm}

We will now deduce Boyd's conjecture \cite[Table 1, $k=3$]{boyd:expmath} for the elliptic curve $E_1$ of conductor $21$ given by the equation $P(x,y)=x+\frac{1}{x}+y+\frac{1}{y}+3=0$.

\begin{thm}\label{boyd 21}
We have the identity $m(x+\frac{1}{x}+y+\frac{1}{y}+3)=2L'(E_1,0)$.
\end{thm}
The change of variables
\begin{equation*}
X=x(x+y+3)+1 \qquad Y=x(x+1)(x+y+3)+1
\end{equation*}
puts $E_1$ in the Weierstrass form $Y^2+XY=X^3+X$. This is the elliptic curve labelled $21a4$ in Cremona's tables \cite{cremona:tables}. The Mordell-Weil group $E_1(\Q)$ is isomorphic to $\Z/4\Z$ and is generated by $P_1=(1,1)$.

The polynomial $P$ satisfies Deninger's conditions \cite[3.2]{deninger:mahler}, so we have
\begin{equation*}
m(P)= \frac{1}{2\pi} \int_{\gamma_P} \eta(x,y)
\end{equation*}
where $\gamma_P$ is the path defined by $\gamma_P = \{(x,y) \in E_1(\C) : |x|=1, |y| \leq 1\}$. The path $\gamma_P$ joins the point $\bar{A} = (\bar{\zeta_3},-1)$ to $A=(\zeta_3,-1)$. Note that these points have last coordinate $-1$, so the discussion in \cite[p. 272]{deninger:mahler} applies and $\gamma_P$ defines an element of $H_1(E_1(\C),\Q)$. After some computation, we find that $\gamma_P = \frac12 \gamma_1$ where $\gamma_1$ is a generator of $H_1(E_1(\C),\Z)^-$ such that $\gamma_{E_1}^+ \bullet \gamma_1 = 2$ (note that $E_1(\R)$ is connected). Using Proposition \ref{pro int eta DE}, it follows that
\begin{equation*}
\int_{\gamma_P} \eta(x,y) = \frac12 \int_{\gamma_1} \eta(x,y) = - D_{E_1}(\beta)
\end{equation*}
where $\beta = \dv(x) * \dv(y)^-$ is the convolution of the divisors of $x$ and $y$. We have
\begin{equation*}
\dv(x) = (P_1)+(2P_1)-(-P_1)-(0) \qquad \dv(y) = (P_1)-(2P_1)-(-P_1)+(0)
\end{equation*}
so that $\beta = 4(P_1)-4(-P_1)$. This gives
\begin{equation*}
\int_{\gamma_P} \eta(x,y) = -8D_{E_1}(P_1).
\end{equation*}

We are now going to relate elliptic dilogarithms on $E_1$ and $E_0$ using Proposition \ref{pro incident} and Lemma \ref{lem DE DE'}. The curve $E_1$ is the $X_1(21)$-optimal elliptic curve in the isogeny class of $E_0$. We have a $2$-isogeny $\lambda : E_1 \to E_0$ whose kernel is generated by $2P_1=(0,0)$. Using Vélu's formulas \cite{velu}, we find that an equation of $\lambda$ is
\begin{equation*}
\lambda(X,Y) = \Bigl(\frac{X^2+1}{X},-\frac{1}{X}+\frac{X^2-1}{X^2} Y\Bigr).
\end{equation*}
The preimages of $P+Q$ under $\lambda$ are the points $A=(\zeta_3,-1-\zeta_3)$ and $\bar{A}=(\bar{\zeta_3},-1-\bar{\zeta_3})$, while the preimages of $P$ are given by $B=(\frac{5+\sqrt{21}}{2},4+\sqrt{21})$ and $B'=(\frac{5-\sqrt{21}}{2},4-\sqrt{21})$. Note that $2A=-P_1$ and $2B=P_1$ so that $A$ and $B$ have order $8$ and we have the relations $\bar{A}=A+2P_1=5A$ and $B'=5B$. Moreover $C=A+B$ is the $2$-torsion point given by $C=(\frac{-1+3i\sqrt{7}}{8},\frac{1-3i\sqrt{7}}{16})$. Using Theorem \ref{zagier 21} and Lemma \ref{lem DE DE'}, we have
\begin{equation*}
L'(E_0,0) = \frac{4}{\pi} \bigl(D_{E_1}(A)+D_{E_1}(\bar{A})-D_{E_1}(B)-D_{E_1}(B')\bigr)
\end{equation*}
so that Theorem \ref{boyd 21} reduces to show
\begin{equation*}
D_{E_1}(P_1)=-2 \bigl(2D_{E_1}(A)-D_{E_1}(B)-D_{E_1}(B')\bigr).
\end{equation*}
We look for lines $\ell$ in $\mathbf{P}^2$ such that $\ell \cap E_1$ is contained in the subgroup generated by $A$ and $B$. Using a computer search, we find that the tangents to $E$ at $A$ and $-A$ and the line $\ell : Y+\frac12 X=0$ passing through the 2-torsion points of $E$ are incident. By Proposition \ref{pro incident}, we deduce the relation
\begin{equation*}
\begin{split}
& 4 D_{E_1}(2A)+4D_{E_1}(3A)+D_{E_1}(4A)+2D_{E_1}(-2A)+4D_{E_1}(-A)\\
& \quad + 2 D_{E_1}(2A+C)+4D_{E_1}(3A+C)+2D_{E_1}(-2A+C)+4D_{E_1}(-A+C)=0.
\end{split}
\end{equation*}
Since $D_{E_1}$ is odd and $D_{E_1}(3A)=-D_{E_1}(\bar{A})=-D_{E_1}(A)$, this simplifies to
\begin{equation*}
2 D_{E_1}(2A)-8D_{E_1}(A)+4D_{E_1}(B)+4D_{E_1}(B')=0
\end{equation*}
which is the desired equality.

%
%
%

\subsection{Conductor 48}

We prove the following case of Boyd's conjectures \cite[Table 1, $k=12$]{boyd:expmath}.

\begin{thm}\label{boyd 48}
We have the identity $m(x+\frac{1}{x}+y+\frac{1}{y}+12)=2L'(E,0)$, where $E$ is the elliptic curve defined by $x+\frac{1}{x}+y+\frac{1}{y}+12=0$.
\end{thm}

The curve $x+\frac{1}{x}+y+\frac{1}{y}+12=0$ is isomorphic to the elliptic curve $E=48a5$. We have a commutative diagram
\begin{equation}\label{cd 48}
\begin{tikzcd}
X_1(48) \arrow{r}{\pi} \arrow{d}{\varphi_1} & X_0(48) \arrow{d}{\varphi_0} \\
E_1 \arrow{r}{\lambda_0} & E_0 \arrow{r}{\lambda} & E.
\end{tikzcd}
\end{equation}
Here $E_1=48a4$ is the $X_1(48)$-optimal elliptic curve and $E_0=48a1$ is the strong Weil curve in the isogeny class of $E$. They are given by the equations
\begin{equation}
E_1 : y^2=x^3+x^2+x \qquad E_0 : y^2=x^3+x^2-4x-4.
\end{equation}
The isogeny $\lambda_0$ has degree $2$ and its kernel is generated by $P_1=(0,0)$. Using Vélu's formulas, we find an explicit equation for $\lambda_0$:
\begin{equation}
\lambda_0(x,y) = \left(x+\frac{1}{x},(1-\frac{1}{x^2})y\right).
\end{equation}
The modular parametrization $\varphi_0$ has degree $2$ and we have
\begin{gather*}
\varphi_0(0)=\varphi_0(1/2)=0 \qquad \varphi_0(1/3)=\varphi_0(1/6)=(-1,0) \\
\varphi_0(1/8)=\varphi_0(1/16)=(-2,0) \qquad \varphi_0(1/24)=\varphi_0(1/48)=(2,0)\\
\varphi_0(1/4)=(0,2i) \qquad \varphi_0(-1/4)=(0,-2i)\\
\varphi_0(1/12)=(-4,-6i) \qquad \varphi_0(-1/12)=(-4,6i).
\end{gather*}
Moreover the ramification indices of $\varphi_0$ at the cusps $\frac14,-\frac14,\frac{1}{12},-\frac{1}{12}$ are equal to $2$. Let $S_0$ be the set of points $P$ of $E_0(\C)$ such that $\varphi_0^{-1}(P)$ is contained in the set of cusps of $X_0(48)$, and similarly let $S_1$ be the set of points $P$ of $E_1(\C)$ such that $\varphi_1^{-1}(P)$ is contained in the set of cusps of $X_1(48)$. By the previous computation, we have
\begin{equation}
S_0 = E_0[2] \cup \{(0,\pm 2i),(-4,\pm 6i)\}.
\end{equation}
The curve $E_0$ doesn't admit a parametrization by modular units, but the curve $E_1$ does. Indeed, consider the point $A=(i,i) \in E_1(\C)$. It has order $8$ and satisfies $\bar{A}=3A$ and $4A=P_1$. Moreover $\lambda_0(A)=(0,2i)$. Because of the commutative diagram (\ref{cd 48}), we know that $S_1$ contains $\lambda_0^{-1}(S_0)$; in particular $S_1$ contains the subgroup generated by $A$. Therefore the following functions on $E_1$ are modular units
\begin{equation}
(f)=2(P_1)-2(0) \qquad (g)=2(A)+2(\bar{A})-4(0).
\end{equation}
We may take $f=x$ and $g=x^2-2y+2x+1$. It is plain that $f$ and $g$ parametrize $E_1$. Moreover the tame symbols of $\{f,g\}$ at $0,P_1,A,\bar{A}$ are equal to $1,1,-1,-1$ so that $\{f,g\}$ belongs to $K_2(E_1) \otimes \Q$. The expression of $f$ and $g$ in terms of Siegel units is
\begin{equation}
\varphi_1^* f = \frac{g_2 g_{20} g_{22}}{g_4 g_{10} g_{14}} \qquad \varphi_1^* g = \frac{g_1^2 g_2 g_{10} g_{11}^2 g_{12}^4 g_{13}^2 g_{14} g_{22} g_{23}^2}{g_4^3 g_5^2 g_6^2 g_7^2 g_{17}^2 g_{18}^2 g_{19}^2 g_{20}}. 
\end{equation}
A generator $\gamma_1$ of $H_1(E_1(\C),\Z)^-$ is given by
\begin{equation*}
\gamma_1 = (\varphi_1)_* \left\{-\frac17,\frac17\right\} = (\varphi_1)_* \left( \xi \begin{pmatrix} 1 & 0 \\ 7 & 1 \end{pmatrix}-\xi \begin{pmatrix} 1 & 0 \\ -7 & 1 \end{pmatrix} \right).
\end{equation*}
Using Theorem \ref{main thm}, we find
\begin{equation}\label{int gamma1}
\int_{\gamma_1} \eta(f,g) = \int_{-1/7}^{1/7} \eta(\varphi_1^* f,\varphi_1^*g) =  \pi L'(F_1,0)
\end{equation}
where $F_1$ is the modular form of weight 2 and level 48 given by
\begin{equation*}
F_1=  4q^2 + 8q^3 - 4q^6 - 8q^{10} - 32q^{11} - 16q^{15} + 4q^{18} + 32q^{19} + \ldots
\end{equation*}
This time $F_1$ is not a multiple of the newform $f_{E_1}$ associated to $E_1$. We look for another modular symbol. Another generator $\gamma_2$ of $H_1(E_1(\C),\Z)^-$ is given by
\begin{equation*}
\gamma_2 = (\varphi_1)_* \left\{-\frac{2}{11},\frac{2}{11}\right\} = (\varphi_1)_* \left( \xi \begin{pmatrix} 2 & 1 \\ 11 & 6 \end{pmatrix} + \xi \begin{pmatrix} 1 & 0 \\ 6 & 1 \end{pmatrix} - \xi \begin{pmatrix} -2 & 1 \\ 11 & -6 \end{pmatrix} -\xi \begin{pmatrix} 1 & 0 \\ -6 & 1 \end{pmatrix} \right).
\end{equation*}
Using Theorem \ref{main thm}, we find
\begin{equation}\label{int gamma2}
\int_{\gamma_2} \eta(f,g) = \int_{-2/11}^{2/11} \eta(\varphi_1^* f,\varphi_1^*g) =  \pi L'(F_2,0)
\end{equation}
where $F_2$ is the modular form of weight 2 and level 48 given by
\begin{equation*}
F_2 =  -4q + 8q^2 + 12q^3 + 8q^5 - 8q^6 - 4q^9 - 16q^{10} - 48q^{11} + 8q^{13} - 24q^{15} - 8q^{17} + 8q^{18} + 48q^{19}+  \ldots
\end{equation*}
A computation reveals that $2F_1-F_2=4f_{E_1}$. Summing (\ref{int gamma1}) and (\ref{int gamma2}), we get
\begin{equation}\label{int gamma12 L}
\int_{2\gamma_1-\gamma_2} \eta(f,g) = 4\pi L'(E_1,0).
\end{equation}
Since $\gamma_{E_1}^+ \bullet \gamma_1 = \gamma_{E_1}^+ \bullet \gamma_2 = 2$, Proposition \ref{pro int eta DE} gives
\begin{equation}\label{int gamma12 D}
\int_{2\gamma_1-\gamma_2} \eta(f,g) = -2 D_{E_1}(\beta(f,g)) = -32 D_{E_1}(A).
\end{equation}
Combining (\ref{int gamma12 L}) and (\ref{int gamma12 D}), we have thus shown Zagier's conjecture for $E_1$.

\begin{thm}\label{zagier 48}
We have the identities $L'(E_1,0)=-\frac{8}{\pi} D_{E_1}(A)$ and $L(E_1,2)=-\frac{2\pi}{3} D_{E_1}(A)$.
\end{thm}

Let us now turn to the elliptic curve $E$. Let $P_k$ be the polynomial $P_k(x,y)=x+1/x+y+1/y+k$. For $k \not\in \{ 0, \pm 4\}$, let $C_k$ be the elliptic curve defined by $P_k(x,y)=0$. The change of variables
\begin{equation*}
X = 4x(x+y+k) \qquad Y=8x^2(x+y+k)
\end{equation*}
puts $C_k$ in Weierstrass form $Y^2+2kXY+8kY=X^3+4X^2$. The point $Q=(0,0)$ on $C_k$ has order $4$. We show that the Mahler measure of $P_k$ can be expressed in terms of the elliptic dilogarithm.

\begin{pro}\label{pro DCk}
Let $k$ be a real number such that $|k| > 4$. We have
\begin{equation*}
m(P_k) = \begin{cases} -\frac{4}{\pi} D_{C_k}(Q) & \textrm{if } k>0,\\
\frac{4}{\pi} D_{C_k}(Q) & \textrm{if } k<0.
\end{cases}
\end{equation*}
\end{pro}

\begin{proof}
Since $|k|>4$, the polynomial $P_k$ doesn't vanish on the torus, so that
\begin{equation*}
m(P_k) = \frac{1}{2\pi} \int_{\gamma_k} \eta(x,y)
\end{equation*}
where $\gamma_k$ is the closed path on $C_k(\C)$ defined by
\begin{equation*}
\gamma_k = \{(x,y) \in C_k(\C) : |x|=1, |y| \leq 1\}.
\end{equation*}
It turns out that $\gamma_k$ is a generator of $H_1(C_k(\C),\Z)^-$ which satisfies $\gamma_{C_k}^+ \bullet \gamma_k = \operatorname{sgn}(k)$. The divisors of $x$ and $y$ are given by
\begin{equation*}
\dv(x) = (Q)+(2Q)-(-Q)-(0) \qquad \dv(y) = (Q)-(2Q)-(-Q)+(0).
\end{equation*}
Since $P_k$ is tempered, we have $\{x,y\} \in K_2(C_k) \otimes \Q$, and Proposition \ref{pro int eta DE} gives
\begin{equation*}
\int_{\gamma_k} \eta(x,y) = -\operatorname{sgn}(k) D_{C_k}(\beta(x,y)) = -8 \operatorname{sgn}(k) D_{C_k}(Q).
\end{equation*}
\end{proof}

\begin{remark}
The fact that $m(P_k)$ can be expressed as an Eisenstein-Kronecker series was also proved by F. Rodriguez-Villegas \cite{rodriguez:modular}.
\end{remark}

We are now going to relate elliptic dilogarithms on $E=C_{12}$ and $E_1$. Let $\lambda' : E_1 \to E$ be the isogeny $\lambda \circ \lambda_0$ from (\ref{cd 48}). It is cyclic of degree 8 and its kernel is generated by the point $B=(-2-\sqrt{3},3i+2i\sqrt{3})$. A preimage of $Q$ under $\lambda'$ is given by
\begin{equation*}
C = \left(\frac12(\alpha^3+\alpha^2+\alpha-1),\frac12(\alpha^3+\alpha^2-\alpha-3)\right)
\end{equation*}
with $\alpha = \sqrt[4]{-3}$. The point $C$ has order 4 and we have $A=B+2C$. By Lemma \ref{lem DE DE'}, we have
\begin{equation}\label{eq DEQ}
D_E(Q) = 2 \sum_{k \in \Z/8\Z} D_{E_1}(C+kB).
\end{equation}
Combining Theorem \ref{zagier 48}, Proposition \ref{pro DCk} and (\ref{eq DEQ}), Theorem \ref{boyd 48} reduces to show
\begin{equation}\label{rel DE 48}
\sum_{k \in \Z/8\Z} D_{E_1}(C+kB) = 2D_{E_1}(A).
\end{equation}
Let $T$ be the subgroup generated by $B$ and $C$. It is isomorphic to $\Z/8\Z \times \Z/4\Z$. There are 187 lines $\ell$ of $\mathbf{P}^2$ such that $\ell \cap E_1$ is contained in $T$. A computer search reveals that among them, there are 691 unordered triples of lines meeting at a point outside $E_1$. These incident lines yield a subgroup $\mathcal{R}$ of $\Z[T]$ of rank 18 such that $D_{E_1}(\mathcal{R})=0$. Let $\mathcal{R}_{\mathrm{triv}}$ be the subgroup of $\Z[T]$ generated by the following elements
\begin{equation}\label{eq fonc}
[P]-[\bar{P}], \qquad [P]+[-P], \qquad [2P]-2 \sum_{Q \in E_1[2]} [P+Q] \qquad (P,Q \in T).
\end{equation}
The group $\mathcal{R}_{\mathrm{triv}}$ has rank 26 and by Lemma \ref{lem DE DE'}, we have $D_{E_1}(\mathcal{R}_{\mathrm{triv}})=0$. Moreover $\mathcal{R}+\mathcal{R}_{\mathrm{triv}}$ has rank 27 and a generator of $(\mathcal{R}+\mathcal{R}_{\mathrm{triv}})/\mathcal{R}_{\mathrm{triv}}$ is given (for example) by the divisor
\begin{equation*}
\beta = \beta_{E_1}(\ell_1,\ell_2)+\beta_{E_1}(\ell_2,\ell_3)+\beta_{E_1}(\ell_3,\ell_1)
\end{equation*}
where $\ell_1$, $\ell_2$, $\ell_3$ are the lines defined by
\begin{align*}
\ell_1 \cap E_1 & = (B)+(-B)+(0)\\
\ell_2 \cap E_1 & = (B+2C)+(B-C)+(-2B-C)\\
\ell_3 \cap E_1 & = (4B+C)+(-3B+2C)+(-B+C).
\end{align*}
Computing explicitly, this gives
\begin{equation*}
\beta = 2 \left(\sum_{k \in \Z/8\Z} (2C+kB)+(3C+kB)\right) - 2(-A)-2(-\bar{A}) + (4B) - (2C) - (4B+2C).
\end{equation*}
Using the functional equations (\ref{eq fonc}) of $D_{E_1}$, we obtain (\ref{rel DE 48}).

\bibliographystyle{smfplain}
\bibliography{references}

\end{document}